\documentclass[10pt,reqno]{amsart}

\usepackage{amssymb}
\usepackage{amsmath}
\usepackage{amsfonts}
\usepackage{amscd}
\usepackage{mathrsfs}
\usepackage{bigints}
\usepackage{cjhebrew}

\newcommand{\C}{\mathbf{C}}
\newcommand{\D}{\mathbf{D}}
\newcommand{\N}{\mathbf{N}}
\newcommand{\Pe}{\mathbf{P}}
\newcommand{\R}{\mathbf{R}}
\newcommand{\Z}{\mathbf{Z}}

\newtheorem{theorem}{Theorem}[section]
\newtheorem{corollary}[theorem]{Corollary}

\newtheorem{example}[theorem]{Example}

\numberwithin{equation}{section}

\begin{document}

\title[PDEs of Waring's-problem form in $\C^n$]{\bf On partial differential equations of Waring's-problem form in several complex variables}

\author[Q. Han]{Qi Han}

\address{Department of Computational, Engineering, and Mathematical Sciences
\vskip 2pt Texas A\&M University-San Antonio, San Antonio, Texas 78224, USA
\vskip 2pt Email: {\sf qhan@tamusa.edu}}



\begin{abstract}
In this paper, we first consider the pseudoprimeness of meromorphic solutions $u$ to a family of partial differential equations (PDEs) $H(u_{z_1},u_{z_2},\ldots,u_{z_n})=P(u)$ of Waring's-problem form, where $H(z_1,z_2,\ldots,z_n)$ is a nontrivial homogenous polynomial of degree $\ell$ in $\C^n$ and $P(w)$ is a polynomial of degree $\hbar$ in $\C$ with all zeros distinct.
Then, we study when these PDEs can admit entire solutions in $\C^n$ and further find these solutions for important cases including particularly $u^\ell_{z_1}+u^\ell_{z_2}+\cdots+u^\ell_{z_n}=u^\hbar$, which are often said to be PDEs of super-Fermat form if $\hbar=0,\ell$ and an eikonal equation if $\ell=2$ and $\hbar=0$.
\end{abstract}

\maketitle

\section{Introduction}\label{Int} 
In this work, we first consider the pseudoprimeness of meromorphic solutions $u$ to a family of partial differential equations $H(u_{z_1},u_{z_2},\ldots,u_{z_n})=P(u)$ of Waring's-problem form, where $H(z_1,z_2,\ldots,z_n)$ is a homogenous polynomial of degree $\ell\hspace{0.2mm}(\geq1)$ in $\C^n$ and $P(w)$ is a polynomial of degree $\hbar$ in $\C$ with all its zeros distinct.
This paper is inspired by Hayman \cite{Ha2,Ha3}; see also Gundersen-Hayman \cite{GH} and several other apposite results discussed later.

Now, let $u$ be a (generic) meromorphic function in $\C^n$.
$u(z)$ is said to admit a factorization $u(z)=f(g(z))$ for a meromorphic left factor $f:\C\to\Pe:=\C\cup\{\infty\}$ and an entire right factor $g:\C^n\to\C$ ($g$ can be meromorphic, provided $f$ is rational).
$u$ is said to be {\sl prime} if all such factorizations lead to either $f$ bilinear or $g$ linear, and $u$ is said to be {\sl pseudoprime} if all such factorizations lead to either $f$ rational or $g$ a polynomial.

The first mathematically rigorous treatment on factorization of meromorphic functions in $\C$ using pseudoprimeness seems to be Gross \cite{Gr}, which was later extended to $\C^n$ by Li-Yang \cite{LY}.
This research topic has found its use in other fields of complex analysis as demonstrated in the work of Bergweiler \cite{Be1,Be2} on normal families and quasiregular maps.
On the other hand, Li \cite{Li1} studied factorization of entire solutions to super-Fermat form partial differential equations in $\C^n$ and proved that all such solutions to $H(u_{z_1},u_{z_2},\ldots,u_{z_n})=1$ are prime; extensions of \cite{Li1} to meromorphic solutions were given by Saleeby \cite{Sa2} and Han \cite{Ha}.

The first main result of this paper considers a general form $P(w)$ that includes those studied in the earlier works as special cases, which is formulated as follows.
(This result seems the first work in literature with general polynomials $H,P$ involved, and indicates a kin relation between solutions to these PDEs and solutions to some well-known ODEs.)

\begin{theorem}\label{Thm1}
Let $u(z)$ be a meromorphic solution in $\C^n$ to the partial differential equation
\begin{equation}\label{Eq1.1}
H(u_{z_1},u_{z_2},\ldots,u_{z_n})=P(u),
\end{equation}
where $H(z_1,z_2,\ldots,z_n)$ is a nontrivial homogenous polynomial of degree $\ell\hspace{0.2mm}(\geq1)$ in $\C^n$ and $P(w)$ is a polynomial of degree $\hbar$ in $\C$ having all zeros distinct.
Then, $u$ is generically pseudoprime.
Moreover, for the cases where $u$ may not be pseudoprime, we have
\vskip 3pt\noindent{\bf Case 1.}
$\ell=\hbar=1$ and $\displaystyle{f(w)=A_1e^{A_0w}+\alpha_1}$;
\vskip 0pt\noindent{\bf Case 2.}
$\ell=1$, $\hbar=2$ and $\displaystyle{f(w)=\frac{\alpha_2A_1e^{A_0(\alpha_1-\alpha_2)w}-\alpha_1}{A_1e^{A_0(\alpha_1-\alpha_2)w}-1}}$;
\vskip 0pt\noindent{\bf Case 3.}
$\ell=\hbar=2$ and $\displaystyle{f(w)=\frac{\alpha_1-\alpha_2}{2}\sin\bigl(\sqrt{-A_0}\hspace{0.2mm}w+A_1\bigr)+\frac{\alpha_1+\alpha_2}{2}}$;
\vskip 0pt\noindent{\bf Case 4.}
$\ell=2$, $\hbar=3$ and $f(w)$ is a transcendental meromorphic solution to
\begin{equation}\label{Eq1.1-1}
(w-a_1)^{m_1}(w-a_2)^{m_2}(f^\prime)^2(w)=A_0\prod^3_{j=1}(f(w)-\alpha_j);
\end{equation}
\vskip 0pt\noindent{\bf Case 5.}
$\ell=2$, $\hbar=4$ and $f(w)$ is a transcendental meromorphic solution to
\begin{equation}\label{Eq1.1-2}
(w-a_1)^{m_1}(w-a_2)^{m_2}(f^\prime)^2(w)=A_0\prod^4_{j=1}(f(w)-\alpha_j).
\end{equation}
\vskip 0pt\noindent Here, $f(w):\C\to\Pe$ is a meromorphic left factor of $u(z)=f(g(z))$ with associated entire right factor $g(z):\C^n\to\C$ transcendental, $m_1,m_2\geq0$ are integers with $m_1+m_2\leq2$, $\alpha_1,\alpha_2,\alpha_3,\alpha_4$ are pairwise distinct complex numbers, and $A_0\cdot A_1\neq0,a_1\neq a_2$ are constants.
\end{theorem}

Discussions of meromorphic solutions $f$ to the ordinary differential equations (ODEs) \eqref{Eq1.1-1} and \eqref{Eq1.1-2} can be found in Bank-Kaufman \cite[Example 5]{BK1} and \cite{BK2}, and Ishizaki-Toda \cite[Section 3]{IT}, where the meromorphic $f$ are closely related to the Weierstrass $\wp$-function.

\begin{corollary}\label{Cor2}
Let $u(z)$ be a meromorphic solution in $\C^n$ to the partial differential equation
\begin{equation}\label{Eq1.2}
H(u_{z_1},u_{z_2},\ldots,u_{z_n})=P(u),
\end{equation}
where $H(z_1,z_2,\ldots,z_n)$ is a nontrivial homogenous polynomial of degree $\ell\hspace{0.2mm}(\geq1)$ in $\C^n$ and $P(w)$ is a polynomial of degree $\hbar$ in $\C$.
If either $\ell=1$ and $P(w)$ has a multiple zero, or $\ell=2$ and $\hbar\geq5$, or $\ell\geq3$ and $P(w)$ has all zeros distinct, then $u$ is pseudoprime.
\end{corollary}

Theorem \ref{Thm1} and Corollary \ref{Cor2} supplement \cite{Li1,Sa2,Ha} on a broader perspective.

Next, we would like to know when equation \eqref{Eq1.1} has entire solutions in $\C^n$ and what these solutions look like: We are only able to describe this with success the general linear form \eqref{Eq1.3} and PDEs of super-Fermat/Waring's-problem form \eqref{Eq1.4}.
The last main results of this paper, Theorems \ref{Thm3} and \ref{Thm8}, with supplemental examples, are formulated as follows.

\begin{theorem}\label{Thm3}
Let $u(z)$ be an entire solution in $\C^n$ to the partial differential equation
\begin{equation}\label{Eq1.3}
(\rho_1u_{z_1}+\rho_2u_{z_2}+\cdots+\rho_nu_{z_n})^\ell=p(u),
\end{equation}
where $\ell\hspace{0.2mm}(\geq1)$ is an integer, $\rho_1,\rho_2,\ldots,\rho_n$ are constants, and $p(w)$ is a (generic) meromorphic function in $\C$.
Then, $p(w)$ must be a polynomial, say, of degree $\hbar$ in $\C$ and
\vskip 3pt\noindent{\bf Case 1.}
$\displaystyle{u(z)=\sqrt[\ell]{c_0}\hspace{0.2mm}(\sigma_1z_1+\sigma_2z_2+\cdots+\sigma_nz_n)+\Phi(z)}$ with $\hbar=0$ and $p(w)=c_0$;
\vskip 0pt\noindent{\bf Case 2.}
$\displaystyle{u(z)=\Bigl(\frac{\ell-\hbar}{\ell}\sqrt[\ell]{c_0}\hspace{0.2mm}(\sigma_1z_1+\sigma_2z_2+\cdots+\sigma_nz_n)+\Phi(z)\Bigr)^{\frac{\ell}{\ell-\hbar}}+a_1}$ with $\hbar<\ell$, $\frac{\ell}{\ell-\hbar}$ being an integer (such as $\ell=\hbar+1$, or $\ell=\hbar+2$ for even $\ell$), and $p(w)=c_0(w-a_1)^\hbar$;
\vskip 0pt\noindent{\bf Case 3.}
$\displaystyle{u(z)=\Phi(z)e^{\sqrt[\ell]{c_0}\hspace{0.2mm}(\sigma_1z_1+\sigma_2z_2+\cdots+\sigma_nz_n)}+a_1}$ with $\hbar=\ell$ and $p(w)=c_0(w-a_1)^\ell$;
\vskip 0pt\noindent{\bf Case 4.}
$\displaystyle{u(z)=\frac{a_1-a_2}{2}\cosh\bigl(\sqrt[\ell]{c_0}\hspace{0.2mm}(\sigma_1z_1+\sigma_2z_2+\cdots+\sigma_nz_n)+\Phi(z)\bigr)+\frac{a_1+a_2}{2}}$ with $\hbar=\ell$ even and $p(w)=c_0(w-a_1)^{\frac{\ell}{2}}(w-a_2)^{\frac{\ell}{2}}$.
\vskip 3pt\noindent Here, $a_1\neq a_2,c_0,\sigma_1,\sigma_2,\ldots,\sigma_n$ are constants and $\Phi(z)$ is an entire function in $\C^n$ such that $\rho_1\sigma_1+\rho_2\sigma_2+\cdots+\rho_n\sigma_n=1$ and $\rho_1\Phi_{z_1}+\rho_2\Phi_{z_2}+\cdots+\rho_n\Phi_{z_n}=0$.
\end{theorem}

Theorem \ref{Thm3} generalizes Li-Saleeby \cite{LS} and Li \cite{Li2} with an easier/shorter proof.

\begin{example}\label{Exm4}
Set $\Phi(z):=\aleph\bigl(\frac{z_2}{\rho_2}-\frac{z_1}{\rho_1},\frac{z_3}{\rho_3}-\frac{z_2}{\rho_2},\ldots,
\frac{z_n}{\rho_n}-\frac{z_{n-1}}{\rho_{n-1}},\frac{z_1}{\rho_1}-\frac{z_n}{\rho_n}\bigr)$ by virtue of an entire function $\aleph(\eta)$ of $\eta\in\C^n$ to see $\rho_1\Phi_{z_1}+\rho_2\Phi_{z_2}+\cdots+\rho_n\Phi_{z_n}=0$.
\end{example}

\begin{example}\label{Exm5}
Set $\Phi_1(z):=\Xi\bigl(\frac{z_2}{\rho_2}-\frac{z_1}{\rho_1},\frac{z_3}{\rho_3}-\frac{z_1}{\rho_1},\ldots,\frac{z_n}{\rho_n}-\frac{z_1}{\rho_1}\bigr)$ to be an entire function of $z\in\C^n$ (by virtue of an entire function $\Xi(\xi)$ of $\xi\in\C^{n-1}$) to see $\rho_1{\Phi_1}_{z_1}+\rho_2{\Phi_1}_{z_2}+\cdots+\rho_n{\Phi_1}_{z_n}=0$.
Likewise, set $\Phi_2(z),\Phi_3(z),\ldots,\Phi_n(z)$ similarly to see $\rho_1{\Phi_j}_{z_1}+\rho_2{\Phi_j}_{z_2}+\cdots+\rho_n{\Phi_j}_{z_n}=0$ for $j=1,2,\ldots,n$.
$\Phi(z)$ can be generated as linear combinations of $\Phi_1,\Phi_2,\ldots,\Phi_n$.
\end{example}

\begin{example}\label{Exm6}
When $n=2k$ is even, set $\Phi(z):=\Upsilon\bigl(\frac{z_2}{\rho_2}-\frac{z_1}{\rho_1},\frac{z_4}{\rho_4}-\frac{z_3}{\rho_3},\ldots,
\frac{z_{2k}}{\rho_{2k}}-\frac{z_{2k-1}}{\rho_{2k-1}}\bigr)$ to see $\rho_1\Phi_{z_1}+\rho_2\Phi_{z_2}+\cdots+\rho_n\Phi_{z_n}=0$ with $\Upsilon(\theta)$ an entire function of $\theta\in\C^k$.
Apparently, other pairwise distinct rearrangements and their linear combinations generate new $\Phi(z)$.
\end{example}

\begin{example}\label{Exm7}
Set $\Phi(z):=f\bigl((n-1)\frac{z_1}{\rho_1}-\frac{z_2}{\rho_2}-\frac{z_3}{\rho_3}-\cdots-\frac{z_n}{\rho_n}\bigr)$ through an entire function $f(w)$ in $\C$ to see $\rho_1\Phi_{z_1}+\rho_2\Phi_{z_2}+\cdots+\rho_n\Phi_{z_n}=0$.
\end{example}

Finally, we describe entire solutions to the partial differential equation
\begin{equation}\label{Eq1.4}
u^\ell_{z_1}+u^\ell_{z_2}+\cdots+u^\ell_{z_n}=u^\hbar,
\end{equation}
which is considered as the most important problem studied in this paper.

When $\ell=2$ and $\hbar=0$, then \eqref{Eq1.4} is a complex $n$-dimensional eikonal equation.
Caffarelli-Crandall \cite{CC} found that linear functions are the only possible global solutions to \eqref{Eq1.4} in $\R^n$ in this case, motivated by an earlier work of Khavinson \cite{Kh} in $\C^2$; see \cite[Remark 2.3]{CC}.
Hemmati \cite{He} and Saleeby \cite{Sa1} provided different proofs of \cite{Kh}.
In $\C^n$ when $n\geq3$, as first described by Johnsson \cite{Jo}, there are indeed nonlinear complex analytic solutions to eikonal equations.
We shall provide more examples in this regard to supplement those well-known works.

Equation \eqref{Eq1.4} for general $u^\hbar$, particularly, $u$ or $u^\ell$, formally relates to the Waring's problem or the super-Fermat problem.
(One should note that what we are interested in here is different from, in a sense, opposite to, those original fundamental issues in number theory.)

\begin{theorem}\label{Thm8}
Assume that $u(z)$ is an entire solution to the partial differential equation \eqref{Eq1.4} in $\C^n$ for integers $\ell\geq1$ and $\hbar\geq0$ with $0\leq\hbar\leq\ell$.
Then, one has
\vskip 3pt\noindent{\bf Case 1.}
$\displaystyle{u(z)=\sigma_1z_1+\sigma_2z_2+\cdots+\sigma_nz_n+\Phi(z)}$ with $\hbar=0$;
\vskip 0pt\noindent{\bf Case 2.}
$\displaystyle{u(z)=\Bigl(\frac{z_1}{2}+c_1\Bigr)^2+\Bigl(\frac{z_2}{2}+c_2\Bigr)^2+\cdots+\Bigl(\frac{z_n}{2}+c_n\Bigr)^2}$ with $\hbar=1$ and $\ell=2$;
\vskip 0pt\noindent{\bf Case 3.}
$\displaystyle{u(z)=\Bigl(\frac{\ell-\hbar}{\ell}(\sigma_1z_1+\sigma_2z_2+\cdots+\sigma_nz_n)+\Phi(z)\Bigr)^{\frac{\ell}{\ell-\hbar}}}$ with $\hbar<\ell$ and $\frac{\ell}{\ell-\hbar}\in\N$;
\vskip 0pt\noindent{\bf Case 4.}
$\displaystyle{u(z)=\Psi(z)e^{\sigma_1z_1+\sigma_2z_2+\cdots+\sigma_nz_n}}$ with $\hbar=\ell$.
\vskip 3pt\noindent Here, $c_1,c_2,\ldots,c_n,\sigma_1,\sigma_2,\ldots,\sigma_n$ are constants and $\Phi(z),\Psi(z)$ are entire functions in $\C^n$ with $\sum^n_{j=1}\sigma^\ell_j=1$, $\sum^n_{j=1}\sum^\ell_{\iota=1}\sigma^{\ell-\iota}_j\Phi^\iota_{z_j}=0$ and $\sum^n_{j=1}\sum^\ell_{\iota=1}(\sigma_j\Psi)^{\ell-\iota}\Psi^\iota_{z_j}=0$.
\end{theorem}

It is easy to see from the proof that $u^{\ell_1}_{z_1}+u^{\ell_2}_{z_2}+\cdots+u^{\ell_n}_{z_n}=1$ has entire solutions as those in {\bf Case 1} above, where $\ell_1,\ell_2,\ldots,\ell_n\geq1$ are integers, not necessarily the same.

\begin{example}\label{Exm9}
Let $u(z):=\frac{2}{7}z_1+\frac{3}{7}z_2+\frac{6}{7}z_3+f(\varpi)$ be entire in $\C^3$ with $f(w)$ entire in $\C$ and
$\varpi:=\frac{1}{2}\bigl(\frac{12-21i}{13}\bigr)^2z^2_1+\frac{1}{2}\bigl(\frac{18+14i}{13}\bigr)^2z^2_2+\frac{1}{2}z^2_3
+\frac{12-21i}{13}\frac{18+14i}{13}z_1z_2-\frac{12-21i}{13}z_1z_3-\frac{18+14i}{13}z_2z_3$.
Then, routine calculations lead to
\begin{equation*}
\left\{\begin{array}{ll}
u_{z_1}(z)=\frac{2}{7}+\frac{12-21i}{13}\bigl(\frac{12-21i}{13}z_1+\frac{18+14i}{13}z_2-z_3\bigr)f'(\varpi) \medskip\\
u_{z_2}(z)=\frac{3}{7}+\frac{18+14i}{13}\bigl(\frac{12-21i}{13}z_1+\frac{18+14i}{13}z_2-z_3\bigr)f'(\varpi) \medskip\\
u_{z_3}(z)=\frac{6}{7}-\bigl(\frac{12-21i}{13}z_1+\frac{18+14i}{13}z_2-z_3\bigr)f'(\varpi)
\end{array}\right.
\end{equation*}
so that $u^2_{z_1}+u^2_{z_2}+u^2_{z_3}=1$.
\end{example}

\begin{example}\label{Exm10}
Let $u(x):=\frac{1}{2}x_1+\frac{2}{3}x_2+\frac{5}{6}x_3+f(\tilde{y})$ for $\tilde{y}:=ax_1+bx_2-x_3$ be differentiable in $\R^3$ with $f(y)$ differentiable in $\R$ to see $u^3_{x_1}+u^3_{x_2}+u^3_{x_3}=1$, where $a=-\frac{16}{9}b+\frac{25}{9}$ and $b$ is the unique real root of the cubic polynomial $91\kappa^3-100\kappa^2+80\kappa-152=0$.
\end{example}

\begin{example}\label{Exm11}
Let $u(z):=f(\varpi_1)\exp\bigl(\frac{2}{7}z_1+\frac{3}{7}z_2+\frac{6}{7}z_3\bigr)$ for $\varpi_1:=\frac{12-21i}{13}z_1+\frac{18+14i}{13}z_2-z_3$ and $\tilde{u}(z):=\tilde{f}(\varpi_2)\exp\bigl(\frac{1}{2}z_1+\frac{2}{3}z_2+\frac{5}{6}z_3\bigr)$ for $\varpi_2:=az_1+bz_2-z_3$ be entire functions in $\C^3$, with $a=-\frac{16}{9}b+\frac{25}{9}$ as above and $b$ a root (real or complex) of $91\kappa^3-100\kappa^2+80\kappa-152=0$.
Then, we have $u^2_{z_1}+u^2_{z_2}+u^2_{z_3}=u^2$ and $\tilde{u}^3_{z_1}+\tilde{u}^3_{z_2}+\tilde{u}^3_{z_3}=u^3$ respectively.
\end{example}

\begin{example}\label{Exm12}
Let $u(z):=\frac{3}{2}z_1-z_2+\frac{3}{2}z_3-\frac{4}{3}z_4+\frac{3}{2}z_5-\frac{5}{3}z_6+\frac{3}{2}z_7+f(\varpi_1,\varpi_2)$ for $\varpi_1:=z_1+iz_3-z_5-iz_7$ and $\varpi_2:=az_2+bz_4-z_6$ be entire in $\C^7$ to observe
\begin{equation*}
u^2_{z_1}+u^3_{z_2}+u^2_{z_3}+u^3_{z_4}+u^2_{z_5}+u^3_{z_6}+u^2_{z_7}=1,
\end{equation*}
where $a,b$ are constants as above and $f(w_1,w_2)$ is an entire function in $\C^2$.
\end{example}

Example \ref{Exm9} provides a `nonlinear' extension to the one from Johnsson \cite{Jo}; see also Li \cite{Li1}.
Example \ref{Exm10} is relevant to Caffarelli-Crandall \cite{CC}, where it is shown $u^2_{x_1}+u^2_{x_2}+\cdots+u^2_{x_n}=1$ has only linear solutions in $\R^n$.
Example \ref{Exm11} is related to Han \cite{Ha} and Li-Ye \cite{LYe}, where it is shown all complex analytic solutions to $u^\ell_{z_1}+u^\ell_{z_2}=u^\ell$ are purely exponential in $\C^2$ for $\ell\geq2$.
Example \ref{Exm12} is formally related to the generalized Fermat equation, for which one can consult Bennett-Mih\u{a}ilescu-Siksek \cite{BMS} for further information (in number theory), while \cite{Li3,Li4,Li6,LYe} described complex analytic solutions to formally related PDEs in $\C^2$.

The remaining of the paper is as follows: Section \ref{PT1CO2} is devoted to the proofs of Theorem \ref{Thm1} and Corollary \ref{Cor2}, Section \ref{PT3} is devoted to that of Theorem \ref{Thm3}, and, after a brief review of the notion of characteristics, Section \ref{PT8} is devoted to that of Theorem \ref{Thm8}.

\section{Proofs of Theorem \ref{Thm1} and Corollary \ref{Cor2}}\label{PT1CO2} 
\begin{proof}[Proof of Theorem \ref{Thm1}]
Let $u(z)=f(g(z))$ be a meromorphic solution to equation \eqref{Eq1.1} in $\C^n$ for entire $g(z):\C^n\to\C$ and meromorphic $f(w):\C\to\Pe$.
Note if $g$ either is a polynomial or is meromorphic, then $u$ is pseudoprime by definition.
Seeing this, assume subsequently $g$ is a transcendental entire function in $\C^n$.
Substitute $u(z)=f(g(z))$ into \eqref{Eq1.1} to have
\begin{equation}\label{Eq2.1}
H(g_{z_1},g_{z_2},\ldots,g_{z_n})=h(g)
\end{equation}
with $h(w):=P(f(w))/(f^\prime(w))^\ell:\C\to\Pe$.
$h$ is rational by Chang-Li-Yang \cite[Theorem 4.1]{CLY}, as $H$ is a polynomial; that is, $P(f)/(f^\prime)^\ell$ is a rational function, say,
\begin{equation}\label{Eq2.2}
h(w)=\frac{P(f)}{(f^\prime)^\ell}(w)=c_0\frac{(w-a_1)^{m_1}(w-a_2)^{m_2}\cdots(w-a_s)^{m_s}}{(w-b_1)^{l_1}(w-b_2)^{l_2}\cdots(w-b_t)^{l_t}}
\end{equation}
for pairwise distinct complex numbers $a_1,a_2,\ldots,a_s,b_1,b_2,\ldots,b_t$, a constant $c_0\neq0$, and integers $m_1,m_2,\ldots,m_s,l_1,l_2,\ldots,l_t\geq0$.

In view of the proof of Han \cite[Pages 282-283]{Ha}, $t=0$ follows.
For completeness, we sketch a proof here.
In fact, combine \eqref{Eq2.1} and \eqref{Eq2.2} to have
\begin{equation}\label{Eq2.3}
H(g_{z_1},g_{z_2},\ldots,g_{z_n})=c_0\frac{(g-a_1)^{m_1}(g-a_2)^{m_2}\cdots(g-a_s)^{m_s}}{(g-b_1)^{l_1}(g-b_2)^{l_2}\cdots(g-b_t)^{l_t}}.
\end{equation}
As the left-hand side is analytic in $\C^n$, $t$ is at most $1$, in which case $g$ assumes its only possible finite Picard value.
Without loss of generality, suppose $t=1$ and $g(z)-b_1=e^{\beta(z)}$ for an entire function $\beta(z):\C^n\to\C$; then, substitute this into \eqref{Eq2.3} to deduce
\begin{equation*}
H(\beta_{z_1},\beta_{z_2},\ldots,\beta_{z_n})=c_0\frac{(e^\beta+b_1-a_1)^{m_1}\cdots(e^\beta+b_1-a_s)^{m_s}}{e^{(\ell+l_1)\beta}},
\end{equation*}
and an application of \cite[Theorem 4.1]{CLY} leads to $g(z)$ a constant.
So, $t=0$.

Now, equation \eqref{Eq2.3} reads
\begin{equation}\label{Eq2.4}
H(g_{z_1},g_{z_2},\ldots,g_{z_n})=c_0(g-a_1)^{m_1}(g-a_2)^{m_2}\cdots(g-a_s)^{m_s}.
\end{equation}
Following the proof of Li \cite[Page 135]{Li2}, we derive from \eqref{Eq2.4} that
\begin{equation*}
m_1+m_2+\cdots+m_s\leq\ell
\end{equation*}
as a straightforward application of the {\sl logarithmic derivative lemma} by Vitter \cite{Vi}.
So, \eqref{Eq2.1}, \eqref{Eq2.2} and \eqref{Eq2.4} combined leads to an ordinary differential equation
\begin{equation}\label{Eq2.5}
(f^\prime)^\ell(w)=\frac{A_0}{(w-a_1)^{m_1}\cdots(w-a_s)^{m_s}}(f(w)-\alpha_1)\cdots(f(w)-\alpha_\hbar)
\end{equation}
of $f(w):\C\to\Pe$ for a constant $A_0\neq0$ and pairwise distinct complex numbers $\alpha_1,\alpha_2,\ldots,\alpha_\hbar$, where  $P(w)=\alpha_0(w-\alpha_1)(w-\alpha_2)\cdots(w-\alpha_\hbar)$ for a constant $\alpha_0\neq0$.

Below, we consider three different cases and their associated subcases.

\vskip 2pt
{\bf Case 1.} $\ell=1$.
In this case, one has $s\leq1$ and accordingly $m_1\leq1$.

\vskip 2pt
{\bf Subcase 1.1.} $\hbar=1$.
In this subcase, equation \eqref{Eq2.5} reads
\begin{equation}\label{Eq2.6}
\frac{f^\prime(w)}{f(w)-\alpha_1}=A_0\hspace{2mm}\mathrm{or}\hspace{2mm}\frac{f^\prime(w)}{f(w)-\alpha_1}=\frac{A_0}{w-a_1}.
\end{equation}
Easy calculations yield
\begin{equation*}
f(w)=A_1e^{A_0w}+\alpha_1\hspace{2mm}\mathrm{or}\hspace{2mm}f(w)=A_1(w-a_1)^{A_0}+\alpha_1,
\end{equation*}
where $A_0\cdot A_1\neq0$ are constants with $A_0\in\Z$ for the latter subcase.
Notice the second subcase implies that $u$ is pseudoprime, since $f$ is rational here.

\vskip 2pt
{\bf Subcase 1.2.} $\hbar=2$.
In this subcase, equation \eqref{Eq2.5} reads
\begin{equation}\label{Eq2.7}
\frac{f^\prime(w)}{(f(w)-\alpha_1)(f(w)-\alpha_2)}=A_0\hspace{2mm}\mathrm{or}\hspace{2mm}\frac{f^\prime(w)}{(f(w)-\alpha_1)(f(w)-\alpha_2)}=\frac{A_0}{w-a_1}.
\end{equation}
Routine calculations lead to
\begin{equation*}
f(w)=\frac{\alpha_2A_1e^{A_0(\alpha_1-\alpha_2)w}-\alpha_1}{A_1e^{A_0(\alpha_1-\alpha_2)w}-1}\hspace{2mm}\mathrm{or}\hspace{2mm}
f(w)=\frac{\alpha_2A_1(w-a_1)^{A_0(\alpha_1-\alpha_2)}-\alpha_1}{A_1(w-a_1)^{A_0(\alpha_1-\alpha_2)}-1},
\end{equation*}
where $A_0\cdot A_1\neq0$ are constants with $A_0(\alpha_1-\alpha_2)\in\Z$ for the latter subcase.
Note the second subcase again implies that $u$ is pseudoprime, as $f$ is rational.

\vskip 2pt
{\bf Subcase 1.3.} $\hbar\geq3$.
In this subcase, with $m_1\leq1$, equation \eqref{Eq2.5} reads
\begin{equation*}
f^\prime(w)=\frac{A_0}{(w-a_1)^{m_1}}(f(w)-\alpha_1)(f(w)-\alpha_2)\cdots(f(w)-\alpha_\hbar).
\end{equation*}
Now, take $w_j$ to be a root of $f(w)-\alpha_j=0$ for $j=1,2,\ldots,\hbar$; a comparison of its multiplicity on both sides implies, say, $w_1=a_1$ (at most).
When $m_1=0$, $f$ is a constant, for it has $\hbar\hspace{0.2mm}(\geq3)$ distinct finite Picard values; when $m_1=1$ but $\hbar\geq4$, the same occurs.
Finally, when $m_1=1$ and $\hbar=3$, then $\frac{f(w)-\alpha_2}{f(w)-\alpha_3}=e^{\gamma(w)}$ for an entire function $\gamma(w):\C\to\C$; so, $f(w)=\frac{\alpha_3e^{\gamma(w)}-\alpha_2}{e^{\gamma(w)}-1}$.
Hence, it is easily seen from the preceding equation that
\begin{equation*}
\frac{(\alpha_2-\alpha_3)\gamma^\prime e^\gamma}{(e^\gamma-1)^2}
=\frac{A_0}{w-a_1}\frac{(\alpha_3-\alpha_2)^2e^\gamma((\alpha_3-\alpha_1)e^\gamma-(\alpha_2-\alpha_1))}{(e^\gamma-1)^3},
\end{equation*}
or equivalently,
\begin{equation*}
\frac{(w-a_1)\gamma^\prime(w)}{A_0(\alpha_3-\alpha_1)(\alpha_3-\alpha_2)}=-\frac{e^{\gamma(w)}-\frac{\alpha_2-\alpha_1}{\alpha_3-\alpha_1}}{e^{\gamma(w)}-1},
\end{equation*}
which leads to $\gamma$, and correspondingly $f$, a constant since $\frac{\alpha_2-\alpha_1}{\alpha_3-\alpha_1}\neq1$.
Thus, $u$ is pseudoprime, as $f$ is a constant when $g$ is transcendental, so that $g$ must be a polynomial.

\vskip 2pt
{\bf Case 2.} $\ell=2$.
In this case, one has $s\leq2$ and accordingly $m_1+m_2\leq2$.

\vskip 2pt
{\bf Subcase 2.1.} $\hbar=1$.
In this subcase, equation \eqref{Eq2.5} reads
\begin{equation*}
(w-a_1)^{m_1}(w-a_2)^{m_2}(f^\prime)^2(w)=A_0(f(w)-\alpha_1).
\end{equation*}
Taking derivative on both sides of the above equation yields
\begin{equation*}
f^{\prime\prime}(w)+\frac{1}{2}\Bigl(\frac{m_1}{w-a_1}+\frac{m_2}{w-a_2}\Bigr)f^\prime(w)=\frac{A_0}{2}\frac{1}{(w-a_1)^{m_1}(w-a_2)^{m_2}}.
\end{equation*}
When $m_1=m_2=0$, one sees that $f$ is a quadratic polynomial, and hence, $u$ is pseudoprime.
When $1\leq m_1+m_2\leq2$, one derives from routine calculations that
\begin{equation}\label{Eq2.8}
f^\prime(w)=\frac{1}{(w-a_1)^{\frac{m_1}{2}}(w-a_2)^{\frac{m_2}{2}}}\biggl(\frac{A_0}{2}\int\frac{1}{{(w-a_1)^{\frac{m_1}{2}}(w-a_2)^{\frac{m_2}{2}}}}dw+C\biggr),
\end{equation}
which does not allow any meromorphic solution $f^\prime$, and accordingly $f$, in $\C$.

\vskip 2pt
{\bf Subcase 2.2.} $\hbar=2$.
In this subcase, equation \eqref{Eq2.5} can be rewritten as
\begin{equation}\label{Eq2.9}
F^2(w)-\frac{1}{A_0}(w-a_1)^{m_1}(w-a_2)^{m_2}(F^\prime)^2(w)=\Bigl(\frac{\alpha_1-\alpha_2}{2}\Bigr)^2
\end{equation}
for $F(w):=f(w)-\frac{\alpha_1+\alpha_2}{2}$.
In view of Theorem 1 (in a general domain $\D\subseteq\C$) and Example 2 of Li \cite{Li5} (see also Liao-Zhang \cite[Theorem 3.1]{LZ}), \eqref{Eq2.9} has no transcendental meromorphic solution $F$ in $\C$ when $1\leq m_1+m_2\leq2$, so that $u$ is pseudoprime as $f$ may be rational.
When $m_1=m_2=0$, we get $F(w)=\frac{\alpha_1-\alpha_2}{2}\sin\bigl(\sqrt{-A_0}\hspace{0.2mm}w+A_1\bigr)$ by Liao-Tang \cite[Theorem 1]{LT}, so that
\begin{equation*}
f(w)=\frac{\alpha_1+\alpha_2}{2}+\frac{\alpha_1-\alpha_2}{2}\sin\bigl(\sqrt{-A_0}\hspace{0.2mm}w+A_1\bigr),
\end{equation*}
where $A_0\neq0,A_1$ are constants.

\vskip 2pt
{\bf Subcase 2.3.} $\hbar=3$.
In this subcase, equation \eqref{Eq2.5} reads
\begin{equation}\label{Eq2.10}
(w-a_1)^{m_1}(w-a_2)^{m_2}(f^\prime)^2(w)=A_0(f(w)-\alpha_1)(f(w)-\alpha_2)(f(w)-\alpha_3).
\end{equation}
When $m_1=m_2=0$, the Weierstrass $\wp$-function is a transcendental meromorphic solution for suitable constants $A_0\cdot\alpha_1\cdot\alpha_2\cdot\alpha_3\neq0$.
When $m_1=m_2=1$ and $a_1=-a_2=2$, Bank-Kaufman \cite[Section 5]{BK1} constructed a solution, as a composite of $\wp$ and {\sl fractional logarithm}, to equation \eqref{Eq2.10}, and they \cite[Theorem]{BK2} further observed transcendental meromorphic solutions to \eqref{Eq2.10} with nonconstant rational coefficients satisfy $T(r,f)=O\bigl(\log^2r\bigr)$ and $T(r,f)\neq o\bigl(\log^2r\bigr)$.

\vskip 2pt
{\bf Subcase 2.4.} $\hbar=4$.
In this subcase, equation \eqref{Eq2.5} reads
\begin{equation}\label{Eq2.11}
(w-a_1)^{m_1}(w-a_2)^{m_2}(f^\prime)^2(w)=A_0(f(w)-\alpha_1)(f(w)-\alpha_2)(f(w)-\alpha_3)(f(w)-\alpha_4).
\end{equation}
Ishizaki-Toda \cite[Section 3]{IT} provided a detailed discussion on this equation.
Note once \eqref{Eq2.11} has a transcendental meromorphic solution, it will then have at least four such solutions.

\vskip 2pt
{\bf Subcase 2.5.} $\hbar\geq5$.
In this subcase, equation \eqref{Eq2.5} reads
\begin{equation*}
(f^\prime)^2(w)=\frac{A_0}{(w-a_1)^{m_1}(w-a_2)^{m_2}}(f(w)-\alpha_1)(f(w)-\alpha_2)\cdots(f(w)-\alpha_\hbar).
\end{equation*}
By virtue of Hayman \cite[Lemma 2.3 and Theorem 3.1]{Ha1}, one has
\begin{equation}\label{Eq2.12}
\begin{split}
\hbar T(r,f)&=T(r,(f-\alpha_1)(f-\alpha_2)\cdots(f-\alpha_\hbar))+O(1)\\
&=T\bigl(r,(w-a_1)^{m_1}(w-a_2)^{m_2}(f^\prime)^2\bigr)+O(1)\\
&\leq2T(r,f^\prime)+O(\log r)\leq(4+\epsilon)T(r,f)+O(\log r)
\end{split}
\end{equation}
for all $r$ outside of a possible set of finite Lebesgue measure with $\epsilon>0$ arbitrarily small, which implies that $f$ is rational, and therefore, $u$ is pseudoprime.

\vskip 2pt
{\bf Case 3.} $\ell\geq3$.
In this case, $s\leq\ell$ and accordingly $m_1+m_2+\cdots+m_\ell\leq\ell$.

\vskip 2pt
{\bf Subcase 3.1.} $\hbar=1$.
In this subcase, one has $Q(w)(f^\prime)^\ell(w)=A_0(f(w)-\alpha_1)$, where $Q(w)$ is a polynomial of degree no larger than $\ell$.
Let $w_1$ be a root of $f(w)-\alpha_1=0$; a comparison of its multiplicity on both sides implies $Q(w_1)=0$, so that $f-\alpha_1$ has only finitely many zeros.
Besides, one sees that $f$ has no pole.
So, $f(w)-\alpha_1=q(w)e^{\delta(w)}$ for an entire function $\delta$ and a polynomial $q$ with $\deg(q)\leq\deg(Q)$.
Routine calculations lead to $e^{(\ell-1)\delta(w)}=\frac{A_0q(w)}{Q(w)(\delta^\prime q+q^\prime)^\ell(w)}$, which implies $\delta,q^\prime$ are constants.
That is, $f$ is linear, and thus, $u$ is prime.

In fact, all zeros of $f-\alpha_1$ are simple and $q$ is a product of distinct linear factors of $Q$.
The form $e^{(\ell-1)\delta}=\frac{A_0q}{Q(\delta^\prime q+q^\prime)^\ell}$ leads to $\delta$ a constant, and then $q/Q,q^\prime$ constants.

\vskip 2pt
{\bf Subcase 3.2.} $\hbar=2$.
In this subcase, one has $Q(w)(f^\prime)^\ell(w)=A_0(f(w)-\alpha_1)(f(w)-\alpha_2)$.
As shown above, $(f-\alpha_1)(f-\alpha_2)$ has only finitely many zeros.
Hence, $\frac{f(w)-\alpha_1}{f(w)-\alpha_2}=r(w)e^{\delta(w)}$ for an entire function $\delta$ and a rational function $r$ whose zeros and poles are from the zeros of $Q(w)$; so, $f(w)=\frac{\alpha_2r(w)e^{\delta(w)}-\alpha_1}{r(w)e^{\delta(w)}-1}$.
Routine calculations then yield
\begin{equation*}
\biggl(\frac{e^{\delta(w)/2}}{r(w)e^{\delta(w)}-1}\biggr)^{2(\ell-1)}=\frac{A_0r(w)}{(\alpha_1-\alpha_2)^{\ell-2}Q(w)(\delta^\prime r+r^\prime)^\ell(w)},
\end{equation*}
which leads to $\delta$ a constant.
That is, $f$ is rational, and thus, $u$ is pseudoprime.

\vskip 2pt
{\bf Subcase 3.3.} $\hbar\geq3$.
Now, $Q(w)(f^\prime)^\ell(w)=A_0(f(w)-\alpha_1)(f(w)-\alpha_2)\cdots(f(w)-\alpha_\hbar)$.
As in the preceding subcase, $(f(w)-\alpha_1)(f(w)-\alpha_2)\cdots(f(w)-\alpha_\hbar)$ has only finitely many zeros.
By Nevanlinna's second fundamental theorem \cite[Chapter 2]{Ha1}, one has
\begin{equation*}
(\hbar-2)T(r,f)\leq\sum_{j=1}^\hbar N\Bigl(r,\frac{1}{f-\alpha_j}\Bigr)+S(r,f)=\epsilon T(r,f)+O(\log r)
\end{equation*}
for all $r$ outside of a possible set of finite Lebesgue measure with $\epsilon>0$ arbitrarily small, which implies that $f$ is rational, and therefore, $u$ is pseudoprime.
\end{proof}

\begin{proof}[Proof of Corollary \ref{Cor2}]
As in the proof of Theorem \ref{Thm1} for equation \eqref{Eq2.5}, one has
\begin{equation}\label{Eq2.13}
(f^\prime)^\ell(w)=\frac{A_0}{(w-a_1)^{m_1}\cdots(w-a_s)^{m_s}}(f(w)-\alpha_1)^{k_1}\cdots(f(w)-\alpha_\mu)^{k_\mu}
\end{equation}
for integers $\mu,k_1,k_2,\ldots,k_\mu\geq0$ and $P(w)=\alpha_0(w-\alpha_1)^{k_1}(w-\alpha_2)^{k_2}\cdots(w-\alpha_\mu)^{k_\mu}$ satisfying $m_1+m_2+\cdots+m_s\leq\ell$, $\hbar=k_1+k_2+\cdots+k_\mu$ and $\max\{k_1,k_2,\ldots,k_\mu\}\geq2$.

Now, we only need to consider two different cases as follows.

\vskip 2pt
{\bf Case 1.} $\ell=1$.
In this case, $m_1\leq1$ with $s\leq1$.

If $\mu=1$, then $\frac{f^\prime(w)}{(f(w)-\alpha_1)^{k_1}}=\frac{A_0}{(w-a_1)^{m_1}}$.
To get a meromorphic $f$, we notice $m_1=0$, $k_1=2$ and $f(w)=-\frac{1}{A_0w+A_1}+\alpha_1$ for two constants $A_0\neq0,A_1$.
So, $u$ is prime.

If $\mu=2$, then $\frac{f^\prime(w)}{(f(w)-\alpha_1)^{k_1}(f(w)-\alpha_2)^{k_2}}=\frac{A_0}{(w-a_1)^{m_1}}$.
Since $(f-\alpha_1)(f-\alpha_2)$ may have $w=a_1$ as its only zero, $\frac{f(w)-\alpha_1}{f(w)-\alpha_2}=r(w)e^{\delta(w)}$ for an entire function $\delta$ and a (reciprocal) linear function $r$; so, $f(w)=\frac{\alpha_2r(w)e^{\delta(w)}-\alpha_1}{r(w)e^{\delta(w)}-1}$.
As in {\bf Subcase 3.2}, using $k_1+k_2\geq3$, upon standard calculations, we see that $f$ is a linear fractional function, and therefore, $u$ is prime.

If $\mu\geq3$, then exactly as in {\bf Subcase 1.3} with $\max\{k_1,k_2,\ldots,k_\mu\}\geq2$, we deduce that $f$ is a constant, and thus, $u$ is pseudoprime because $g$ must be a polynomial.

In summary, $u$ is pseudoprime when $\ell=1$ and $\max\{k_1,k_2,\ldots,k_\mu\}\geq2$.

\vskip 2pt
{\bf Case 2.} $\ell=2$.
In this case, we can utilize exactly the same analysis as in \eqref{Eq2.12} to have $f$ rational, and therefore, $u$ pseudoprime, provided $\hbar=\deg(P)\geq5$.

On the other hand, notice when $\mu=3$, $k_1=k_2=1$ and $k_3=2$, we have
\begin{equation}\label{Eq2.14}
(w-a_1)^{m_1}(w-a_2)^{m_2}(f^\prime)^2(w)=A_0(f(w)-\alpha_1)(f(w)-\alpha_2)(f(w)-\alpha_3)^2.
\end{equation}
Ishizaki-Toda \cite[Section 2]{IT} provided a detailed discussion on this equation.
Note once \eqref{Eq2.14} has a transcendental meromorphic solution, it will then have at least two such solutions.
\end{proof}

\section{Proof of Theorem \ref{Thm3}}\label{PT3} 
\begin{proof}[Proof of Theorem \ref{Thm3}]
Let $u(z)$ be an entire solution to equation \eqref{Eq1.3} in $\C^n$.
An application of \cite[Theorem 4.1]{CLY} implies that $p(w):\C\to\Pe$ must be a rational function, say,
\begin{equation*}
p(w)=c_0\frac{(w-a_1)^{m_1}(w-a_2)^{m_2}\cdots(w-a_s)^{m_s}}{(w-b_1)^{l_1}(w-b_2)^{l_2}\cdots(w-b_t)^{l_t}}
\end{equation*}
for pairwise distinct complex numbers $a_1,a_2,\ldots,a_s,b_1,b_2,\ldots,b_t$, a constant $c_0\neq0$, and integers $m_1,m_2,\ldots,m_s,l_1,l_2,\ldots,l_t\geq0$.
Therefore, one has
\begin{equation}\label{Eq3.1}
(\rho_1u_{z_1}+\rho_2u_{z_2}+\cdots+\rho_nu_{z_n})^\ell=c_0\frac{(u-a_1)^{m_1}(u-a_2)^{m_2}\cdots(u-a_s)^{m_s}}{(u-b_1)^{l_1}(u-b_2)^{l_2}\cdots(u-b_t)^{l_t}}.
\end{equation}
As the left-hand side is analytic in $\C^n$, $t$ is at most $1$, in which case $u$ assumes its only possible finite Picard value.
Without loss of generality, suppose $t=1$ and $u(z)-b_1=e^{\beta(z)}$ for an entire function $\beta(z):\C^n\to\C$; then, substitute this into \eqref{Eq3.1} to deduce
\begin{equation*}
(\rho_1\beta_{z_1}+\rho_2\beta_{z_2}+\cdots+\rho_n\beta_{z_n})^\ell=c_0\frac{(e^\beta+b_1-a_1)^{m_1}\cdots(e^\beta+b_1-a_s)^{m_s}}{e^{(\ell+l_1)\beta}},
\end{equation*}
and an application of \cite[Theorem 4.1]{CLY} leads to $u(z)$ a constant.
So, $t=0$.

We have shown that $p(w):\C\to\C$ is a polynomial; so, equation \eqref{Eq3.1} reads
\begin{equation}\label{Eq3.2}
(\rho_1u_{z_1}+\rho_2u_{z_2}+\cdots+\rho_nu_{z_n})^\ell=c_0(u-a_1)^{m_1}(u-a_2)^{m_2}\cdots(u-a_s)^{m_s}.
\end{equation}
By virtue of the {\sl logarithmic derivative lemma} (see \cite{Vi}), one immediately sees
\begin{equation}\label{Eq3.3}
\hbar=m_1+m_2+\cdots+m_s\leq\ell.
\end{equation}

Now, let $z_{a_j}\in\C^n$ be a root of $u(z_{a_j})-a_j=0$ with multiversity $\nu^{a_j}_u\in\N$.
Then, it is clear $\min\bigl\{\nu^{a_j}_{u_{z_1}},\nu^{a_j}_{u_{z_2}},\ldots,\nu^{a_j}_{u_{z_n}}\bigr\}=\nu^{a_j}_u-1$.
Using \eqref{Eq3.2}, we also note that $\ell\cdot\nu^{a_j}_{\rho_1u_{z_1}+\rho_2u_{z_2}+\cdots+\rho_nu_{z_n}}=m_j\cdot\nu^{a_j}_u$.
By \eqref{Eq3.2} and \eqref{Eq3.3}, one has either $s=1$, $m_1=\ell$ and $\nu^{a_j}_{\rho_1u_{z_1}+\rho_2u_{z_2}+\cdots+\rho_nu_{z_n}}=\nu^{a_j}_u$, or $s\geq1$, $m_j<\ell$ and $\nu^{a_j}_{\rho_1u_{z_1}+\rho_2u_{z_2}+\cdots+\rho_nu_{z_n}}=\nu^{a_j}_u-1$ so that
\begin{equation}\label{Eq3.4}
\frac{\ell}{2}\leq m_j=\ell\cdot\frac{\nu^{a_j}_u-1}{\nu^{a_j}_u}<\ell
\end{equation}
provided $\nu^{a_j}_u\geq2$.
If $\nu^{a_j}_u=1$, then $\nu^{a_j}_{\rho_1u_{z_1}+\rho_2u_{z_2}+\cdots+\rho_nu_{z_n}}=1$, $s=1$ and $m_1=\ell$.

Next, assume $a_1$ is the finite Picard value of $u$; so, $u(z)-a_1=e^{\gamma(z)}$ for an entire function $\gamma(z):\C^n\to\C$, and \eqref{Eq3.2} reads
\begin{equation*}
(\rho_1\gamma_{z_1}+\rho_2\gamma_{z_2}+\cdots+\rho_n\gamma_{z_n})^\ell=c_0\frac{(e^\gamma+a_1-a_2)^{m_2}\cdots(e^\gamma+a_1-a_s)^{m_s}}{e^{(\ell-m_1)\gamma}},
\end{equation*}
immediately implying $s=1$ and $m_1=\ell$ in view of \cite[Theorem 4.1]{CLY}.
If $s=2$, then none of $a_j$ can be the finite Picard value of $u$, and \eqref{Eq3.3} and \eqref{Eq3.4} lead to $m_1=m_2=\frac{\ell}{2}$.

In summary, one has either $s=1$ and $m_1\leq\ell$, or $s=2$ and $m_1=m_2=\frac{\ell}{2}$.

Below, we consider three different cases and their associated subcases.

\vskip 2pt
{\bf Case 1.} $s=0$.
In this case, one has $p(w)=c_0$ and
\begin{equation}\label{Eq3.5}
\rho_1u_{z_1}+\rho_2u_{z_2}+\cdots+\rho_nu_{z_n}=\sqrt[\ell]{c_0}.
\end{equation}
The characteristic curve of \eqref{Eq3.5} (see Evans \cite[Section 3.2]{Ev}), for a parameter $\tau$, reads
\begin{equation*}
\frac{dz_1}{d\tau}=\rho_1,\,\,\frac{dz_2}{d\tau}=\rho_2,\,\,\ldots,\,\,\frac{dz_n}{d\tau}=\rho_n\,\,\mathrm{and}\,\,\frac{du}{d\tau}=\sqrt[\ell]{c_0}.
\end{equation*}
Given initial conditions, say, $z_1=0,z_2=d_2,\ldots,z_n=d_n$ and $u=\varphi(d_2,\ldots,d_n)$, one has
\begin{equation*}
\begin{split}
&z_1=\rho_1\tau,\,\,z_2=\rho_2\tau+d_2,\,\,\ldots,\,\,z_n=\rho_n\tau+d_n,\\
&\tau=\frac{z_1}{\rho_1},\,\,d_2=z_2-\frac{\rho_2}{\rho_1}z_1,\,\,\ldots,\,\,d_n=z_n-\frac{\rho_n}{\rho_1}\tau,
\end{split}
\end{equation*}
and
\begin{equation*}
\begin{split}
u&=\sqrt[\ell]{c_0}\hspace{0.2mm}\tau+\varphi(d_2,\ldots,d_n)=\sqrt[\ell]{c_0}\hspace{0.2mm}(\varrho_1\tau+\varrho_2\tau+\cdots+\varrho_n\tau)+\varphi(d_2,\ldots,d_n)\\
&=\sqrt[\ell]{c_0}\Bigl(\frac{\varrho_1}{\rho_1}z_1+\frac{\varrho_2}{\rho_2}z_2+\cdots+\frac{\varrho_n}{\rho_n}z_n\Bigr)
-\sqrt[\ell]{c_0}\Bigl(\frac{\varrho_2}{\rho_2}d_2+\cdots+\frac{\varrho_n}{\rho_n}d_n\Bigr)+\varphi(d_2,\ldots,d_n)\\
&=\sqrt[\ell]{c_0}\Bigl(\frac{\varrho_1}{\rho_1}z_1+\frac{\varrho_2}{\rho_2}z_2+\cdots+\frac{\varrho_n}{\rho_n}z_n\Bigr)+\psi(d_2,\ldots,d_n)\\
&=\sqrt[\ell]{c_0}\Bigl(\frac{\varrho_1}{\rho_1}z_1+\frac{\varrho_2}{\rho_2}z_2+\cdots+\frac{\varrho_n}{\rho_n}z_n\Bigr)
+\psi\Bigl(z_2-\frac{\rho_2}{\rho_1}z_1,\ldots,z_n-\frac{\rho_n}{\rho_1}z_1\Bigr)
\end{split}
\end{equation*}
with $\varrho_1,\varrho_2,\ldots,\varrho_n$ complex numbers satisfying $\varrho_1+\varrho_2+\cdots+\varrho_n=1$.
That is,
\begin{equation*}
u(z)=\sqrt[\ell]{c_0}\hspace{0.2mm}(\sigma_1z_1+\sigma_2z_2+\cdots+\sigma_nz_n)+\Phi(z).
\end{equation*}

It is noteworthy that different initial conditions generate different $\Phi(z)$ as those in Example \ref{Exm5}, and there are other $\Phi(z)$ as described in Examples \ref{Exm4}, \ref{Exm6}, \ref{Exm7} and many more.

\vskip 2pt
{\bf Case 2.} $s=1$.
In this case, one has $p(w)=c_0(w-a_1)^\hbar$ and
\begin{equation*}
(\rho_1u_{z_1}+\rho_2u_{z_2}+\cdots+\rho_nu_{z_n})^\ell=c_0(u-a_1)^\hbar.
\end{equation*}

\vskip 2pt
{\bf Subcase 2.1.} $\hbar<\ell$.
In this subcase, we have
\begin{equation}\label{Eq3.6}
\rho_1u_{z_1}+\rho_2u_{z_2}+\cdots+\rho_nu_{z_n}=\sqrt[\ell]{c_0}\hspace{0.2mm}(u-a_1)^{\frac{\hbar}{\ell}}
\end{equation}
with $(u-a_1)^{\frac{\hbar}{\ell}}$ entire in $\C^n$.
The characteristic curve of \eqref{Eq3.6}, for a parameter $\tau$, reads
\begin{equation*}
\frac{dz_1}{d\tau}=\rho_1,\,\,\frac{dz_2}{d\tau}=\rho_2,\,\,\ldots,\,\,\frac{dz_n}{d\tau}=\rho_n\,\,\mathrm{and}\,\,
\frac{du}{d\tau}=\sqrt[\ell]{c_0}\hspace{0.2mm}(u-a_1)^{\frac{\hbar}{\ell}}.
\end{equation*}
Given initial conditions $z_1=d_1,z_2=d_2,\ldots,z_n=d_n$ and $u=\varphi(d_1,d_2,\ldots,d_n)$, one has
\begin{equation*}
(u-a_1)^{1-\frac{\hbar}{\ell}}=\frac{\ell-\hbar}{\ell}\sqrt[\ell]{c_0}\hspace{0.2mm}(\varrho_1\tau+\varrho_2\tau+\cdots+\varrho_n\tau)+\tilde{\varphi}(d_1,d_2,\ldots,d_n)
\end{equation*}
with $\tilde{\varphi}(d_1,d_2,\ldots,d_n):=(\varphi(d_1,d_2,\ldots,d_n)-a_1)^{1-\frac{\hbar}{\ell}}$ entire in $\C^n$ for appropriate $\varphi$.
So,
\begin{equation*}
u(z)=a_1+\Bigl(\frac{\ell-\hbar}{\ell}\sqrt[\ell]{c_0}\hspace{0.2mm}(\sigma_1z_1+\sigma_2z_2+\cdots+\sigma_nz_n)+\Phi(z)\Bigr)^{\frac{\ell}{\ell-\hbar}}.
\end{equation*}
It is apparent that $u(z)$ is an entire function in $\C^n$ when $\frac{\ell}{\ell-\hbar}$ is an integer.

\vskip 2pt
{\bf Subcase 2.2.} $\hbar=\ell$.
In this subcase, we have
\begin{equation}\label{Eq3.7}
\rho_1u_{z_1}+\rho_2u_{z_2}+\cdots+\rho_nu_{z_n}=\sqrt[\ell]{c_0}\hspace{0.2mm}(u-a_1).
\end{equation}
Similarly, the characteristic curve of \eqref{Eq3.7}, for a parameter $\tau$, reads
\begin{equation*}
\frac{dz_1}{d\tau}=\rho_1,\,\,\frac{dz_2}{d\tau}=\rho_2,\,\,\ldots,\,\,\frac{dz_n}{d\tau}=\rho_n\,\,\mathrm{and}\,\,\frac{du}{d\tau}=\sqrt[\ell]{c_0}\hspace{0.2mm}(u-a_1).
\end{equation*}
Given initial conditions $z_1=d_1,z_2=d_2,\ldots,z_n=d_n$ and $u=\varphi(d_1,d_2,\ldots,d_n)$, one has
\begin{equation*}
u-a_1=(\varphi(d_1,d_2,\ldots,d_n)-a_1)\exp\bigl(\sqrt[\ell]{c_0}\hspace{0.2mm}(\varrho_1\tau+\varrho_2\tau+\cdots+\varrho_n\tau)\bigr).
\end{equation*}
That is,
\begin{equation*}
u(z)=a_1+\Phi(z)\exp\bigl(\sqrt[\ell]{c_0}\hspace{0.2mm}(\sigma_1z_1+\sigma_2z_2+\cdots+\sigma_nz_n)\bigr).
\end{equation*}
When $0$ is the finite Picard value of $\Phi(z)$, $a_1$ is the finite Picard value of $u(z)$.
So, if we write $\Phi^*(z):=\ln(\Phi(z))$ to be an entire function in $\C^n$, then it follows that
\begin{equation*}
u(z)=a_1+\exp\bigl(\sqrt[\ell]{c_0}\hspace{0.2mm}(\sigma_1z_1+\sigma_2z_2+\cdots+\sigma_nz_n)+\Phi^*(z)\bigr).
\end{equation*}

\vskip 2pt
{\bf Case 3.} $s=2$.
In this case, one has $p(w)=c_0(w-a_1)^{\frac{\ell}{2}}(w-a_2)^{\frac{\ell}{2}}$ and
\begin{equation}\label{Eq3.8}
(\rho_1u_{z_1}+\rho_2u_{z_2}+\cdots+\rho_nu_{z_n})^2=\sqrt[\ell]{c^2_0}\hspace{0.2mm}(u-a_1)(u-a_2),
\end{equation}
which can be easily rewritten as
\begin{equation*}
\Bigl(\sqrt[\ell]{c_0}\Bigl(u-\frac{a_1+a_2}{2}\Bigr)\Bigr)^2-(\rho_1u_{z_1}+\rho_2u_{z_2}+\cdots+\rho_nu_{z_n})^2=\Bigl(\sqrt[\ell]{c_0}\Bigl(\frac{a_1-a_2}{2}\Bigr)\Bigr)^2.
\end{equation*}
Recalling $u$ is an entire function in $\C^n$, we deduce that
\begin{equation*}
\sqrt[\ell]{c_0}\Bigl(u-\frac{a_1+a_2}{2}\Bigr)\pm(\rho_1u_{z_1}+\rho_2u_{z_2}+\cdots+\rho_nu_{z_n})=\sqrt[\ell]{c_0}\Bigl(\frac{a_1-a_2}{2}\Bigr)e^{\pm\delta(z)}
\end{equation*}
for an entire function $\delta(z):\C^n\to\C$.
As a consequence, one observes
\begin{equation}\label{Eq3.9}
\begin{split}
&u=\frac{a_1+a_2}{2}+\frac{a_1-a_2}{2}\cosh(\delta)\,\,\mathrm{and}\\
&\rho_1u_{z_1}+\rho_2u_{z_2}+\cdots+\rho_nu_{z_n}=\sqrt[\ell]{c_0}\hspace{0.2mm}\frac{a_1-a_2}{2}\sinh(\delta),
\end{split}
\end{equation}
implying
\begin{equation*}
\sinh(\delta)(\rho_1\delta_{z_1}+\rho_2\delta_{z_2}+\cdots+\rho_n\delta_{z_n})=\sqrt[\ell]{c_0}\hspace{0.2mm}\sinh(\delta),
\end{equation*}
which leads back to equation \eqref{Eq3.5} now satisfied by $\delta(z)$.
So, \eqref{Eq3.9} yields
\begin{equation*}
u(z)=\frac{a_1+a_2}{2}+\frac{a_1-a_2}{2}\cosh\bigl(\sqrt[\ell]{c_0}\hspace{0.2mm}(\sigma_1z_1+\sigma_2z_2+\cdots+\sigma_nz_n)+\Phi(z)\bigr).
\end{equation*}

All the preceding discussions conclude the proof of Theorem \ref{Thm3}.
\end{proof}

\section{Proof of Theorem \ref{Thm8}}\label{PT8} 
We start this final section by first briefly reviewing the concept of characteristics following Evans \cite[Section 3.2]{Ev} with symbols adapted to our setting.

Given a general first-order PDE $F(Du,u,z)=0$, for a parameter $\tau$, write
\begin{equation*}
\left\{\begin{array}{ll}
z(\tau):=(z_1(\tau),z_2(\tau),\ldots,z_n(\tau)), \medskip\\
u(\tau):=u(z(\tau))\,\,\mathrm{and} \medskip\\
Du(\tau):=(u_{z_1}(z(\tau)),u_{z_2}(z(\tau)),\ldots,u_{z_n}(z(\tau))).
\end{array}\right.
\end{equation*}
The associated {\sl characteristics}, in terms of $F(x_1,x_2,\ldots,x_n,u,y_1,y_2,\ldots,y_n)$, read
\begin{equation}\label{Eq4.1}
\frac{dz(\tau)}{d\tau}=\Bigl(\frac{dz_1(\tau)}{d\tau},\frac{dz_2(\tau)}{d\tau},\ldots,\frac{dz_n(\tau)}{d\tau}\Bigr)=F_x(Du(\tau),u(\tau),z(\tau))
\end{equation}
with $F_x(x_1,x_2,\ldots,x_n,u,y_1,y_2,\ldots,y_n):=(F_{x_1},F_{x_2},\ldots,F_{x_n})$,
\begin{equation}\label{Eq4.2}
\begin{split}
\frac{dDu(\tau)}{d\tau}&=\Bigl(\frac{du_{z_1}(z(\tau))}{d\tau},\frac{du_{z_2}(z(\tau))}{d\tau},\ldots,\frac{du_{z_n}(z(\tau))}{d\tau}\Bigr)\\
&=-F_u(Du(\tau),u(\tau),z(\tau))Du(\tau)-F_y(Du(\tau),u(\tau),z(\tau))
\end{split}
\end{equation}
with $F_y(x_1,x_2,\ldots,x_n,u,y_1,y_2,\ldots,y_n):=(F_{y_1},F_{y_2},\ldots,F_{y_n})$, and
\begin{equation}\label{Eq4.3}
\frac{du(\tau)}{d\tau}=Du(\tau)\cdot\frac{dz(\tau)}{d\tau}=Du(\tau)\cdot F_x(Du(\tau),u(\tau),z(\tau)).
\end{equation}

Equation \eqref{Eq4.1} is the key to the success of characteristics, and if $F(Du,u,z)=0$ is linear as in the situation of Theorem \ref{Thm3}, then only equations \eqref{Eq4.1} and \eqref{Eq4.3} are needed.

\begin{proof}[Proof of Theorem \ref{Thm8}]
For equation \eqref{Eq1.4}, it is readily seen that $\hbar\leq\ell$ using the same analysis as before and its associated characteristics are simplified to be
\begin{equation}\label{Eq4.4}
\begin{split}
&\frac{dz(\tau)}{d\tau}=\bigl(\ell u^{\ell-1}_{z_1}(z(\tau)),\ell u^{\ell-1}_{z_2}(z(\tau)),\ldots,\ell u^{\ell-1}_{z_n}(z(\tau))\bigr),\\
&\frac{dDu(\tau)}{d\tau}=\hbar u^{\hbar-1}Du(\tau)\,\,\mathrm{and}\,\,\frac{du(\tau)}{d\tau}=\ell u^\hbar.
\end{split}
\end{equation}

Below, we consider four different cases to finish our discussions.

\vskip 2pt
{\bf Case 1.} $\hbar=0$.
In this case, we further consider the partial differential equation
\begin{equation}\label{Eq4.5}
u^{\ell_1}_{z_1}+u^{\ell_2}_{z_2}+\cdots+u^{\ell_n}_{z_n}=1
\end{equation}
with $\ell_1,\ell_2,\ldots,\ell_n\geq1$ integers, not necessarily the same.
As now $\frac{dDu(\tau)}{d\tau}=0$ independent of $\tau$, we write $u_{z_j}(z(\tau))=\sigma_j$ for $j=1,2,\ldots,n$ and $\frac{dz(\tau)}{d\tau}=\bigl(\ell_1\sigma_1^{\ell_1-1},\ell_2\sigma_2^{\ell_2-1},\ldots,\ell_n\sigma_n^{\ell_n-1}\bigr)$.
Given initial conditions, say, $z_1=0,z_2=d_2,\ldots,z_n=d_n$ and $u=\varphi(d_2,\ldots,d_n)$, one has
\begin{equation*}
z_1=\ell_1\sigma_1^{\ell_1-1}\tau,\,\,z_2=\ell_2\sigma_2^{\ell_2-1}\tau+d_2,\,\,\ldots,\,\,z_n=\ell_n\sigma_n^{\ell_n-1}\tau+d_n
\end{equation*}
and
\begin{equation*}
\begin{split}
u&=\bigl(\ell_1\sigma_1^{\ell_1}+\ell_2\sigma_2^{\ell_2}+\cdots+\ell_n\sigma_n^{\ell_n}\bigr)\tau+\varphi(d_2,\ldots,d_n)\\
&=\bigl(\ell_1\sigma_1^{\ell_1}+\ell_2\sigma_2^{\ell_2}+\cdots+\ell_n\sigma_n^{\ell_n}\bigr)(\varrho_1\tau+\varrho_2\tau+\cdots+\varrho_n\tau)+\varphi(d_2,\ldots,d_n)\\
&=\bigl(\ell_1\sigma_1^{\ell_1}+\ell_2\sigma_2^{\ell_2}+\cdots+\ell_n\sigma_n^{\ell_n}\bigr)
\biggl(\frac{\varrho_1z_1}{\ell_1\sigma_1^{\ell_1-1}}+\frac{\varrho_2z_2}{\ell_2\sigma_2^{\ell_2-1}}+\cdots+\frac{\varrho_nz_n}{\ell_n\sigma_n^{\ell_n-1}}\biggr)+\Phi(z)
\end{split}
\end{equation*}
following Theorem \ref{Thm3}, {\bf Case 1} verbatim for constants $\varrho_1,\varrho_2,\ldots,\varrho_n$ with $\varrho_1+\varrho_2+\cdots+\varrho_n=1$.
Take $\varrho_j:=\frac{\ell_j\sigma^{\ell_j}_j}{\ell_1\sigma_1^{\ell_1}+\ell_2\sigma_2^{\ell_2}+\cdots+\ell_n\sigma_n^{\ell_n}}$ for $j=1,2,\ldots,n$ to deduce
\begin{equation*}
u(z)=\sigma_1z_1+\sigma_2z_2+\cdots+\sigma_nz_n+\Phi(z)
\end{equation*}
with $\sigma^{\ell_1}_1+\sigma^{\ell_2}_2+\cdots+\sigma^{\ell_n}_n=1$ and $\sum^n_{j=1}\sum^{\ell_j}_{\iota=1}\sigma^{\ell_j-\iota}_j\Phi^\iota_{z_j}=0$.

\vskip 2pt
{\bf Case 2.} $\hbar=1$.
In this case, we further consider the partial differential equation
\begin{equation}\label{Eq4.6}
u^{\ell_1}_{z_1}+u^{\ell_2}_{z_2}+\cdots+u^{\ell_n}_{z_n}=u.
\end{equation}
The second equation in \eqref{Eq4.4} now reads $\frac{dDu(\tau)}{d\tau}=Du(\tau)$, and thus,
\begin{equation}\label{Eq4.7}
Du(\tau)=(u_{z_1}(z(\tau)),u_{z_2}(z(\tau)),\ldots,u_{z_n}(z(\tau)))=(\varsigma_1e^\tau,\varsigma_2e^\tau,\ldots,\varsigma_ne^\tau)
\end{equation}
with $\varsigma_j:=u_{z_j}(z(0))$ for $j=1,2,\ldots,n$.
Consequently, this leads to
\begin{equation*}
\begin{split}
\frac{dz(\tau)}{d\tau}&=\bigl(\ell_1u^{\ell_1-1}_{z_1}(z(\tau)),\ell_2u^{\ell_2-1}_{z_2}(z(\tau)),\ldots,\ell_nu^{\ell_n-1}_{z_n}(z(\tau))\bigr)\\
&=\bigl(\ell_1\varsigma^{\ell_1-1}_1e^{(\ell_1-1)\tau},\ell_2\varsigma^{\ell_2-1}_2e^{(\ell_2-1)\tau},\ldots,\ell_n\varsigma^{\ell_n-1}_ne^{(\ell_n-1)\tau}\bigr),
\end{split}
\end{equation*}
so that
\begin{equation}\label{Eq4.8}
z_j(\tau)=\frac{\ell_j}{\ell_j-1}\varsigma^{\ell_j-1}_j\bigl(e^{(\ell_j-1)\tau}-1\bigr)+d_j
\end{equation}
with $d_j:=z_j(0)$ for $j=1,2,\ldots,n$.
Finally, by \eqref{Eq4.3}, we have
\begin{equation*}
\frac{du(\tau)}{d\tau}=\ell_1\varsigma^{\ell_1}_1e^{\ell_1\tau}+\ell_2\varsigma^{\ell_2}_2e^{\ell_2\tau}+\cdots+\ell_n\varsigma^{\ell_n}_ne^{\ell_n\tau},
\end{equation*}
and thus, for $u_0:=u(z(0))=\varphi(d_1,d_2,\ldots,d_n)$, one observes
\begin{equation}\label{Eq4.9}
u(\tau)=\varsigma^{\ell_1}_1\bigl(e^{\ell_1\tau}-1\bigr)+\varsigma^{\ell_2}_2\bigl(e^{\ell_2\tau}-1\bigr)+\cdots+\varsigma^{\ell_n}_n\bigl(e^{\ell_n\tau}-1\bigr)+u_0.
\end{equation}
Combine \eqref{Eq4.8} and \eqref{Eq4.9} with routine calculations to deduce
\begin{equation*}
u(z)=\sum^n_{j=1}\varsigma^{\ell_j}_j\biggl(\frac{(z_j-d_j)(\ell_j-1)}{\ell_j\varsigma^{\ell_j-1}_j}+1\biggr)^{\frac{\ell_j}{\ell_j-1}}
\end{equation*}
as $\varsigma^{\ell_1}_1+\varsigma^{\ell_2}_2+\cdots+\varsigma^{\ell_n}_n=u_0$ by \eqref{Eq4.6}, \eqref{Eq4.7} and \eqref{Eq4.9}.
To have $u$ entire, it must be $\ell_1=\ell_2=\cdots=\ell_n=2$; therefore,
\begin{equation}\label{Eq4.10}
u(z)=\frac{z^2_1}{4}+\frac{z^2_2}{4}+\cdots+\frac{z^2_n}{4}+\Lambda(z),
\end{equation}
where $\Lambda(z)$ is an entire function in $\C^n$ depending on $d_j,\varsigma_j$ for $j=1,2,\ldots,n$.

Below, we show $\Lambda(z)$ is linear.
In fact, $u$ being an entire solution to \eqref{Eq4.6} implies
\begin{equation}\label{Eq4.11}
\sum^n_{j=1}\bigl(z_j\Lambda_{z_j}(z)+\Lambda^2_{z_j}(z)\bigr)-\Lambda(z)=0.
\end{equation}
Equation \eqref{Eq4.2} immediately yields $\frac{dD\Lambda(\tau)}{d\tau}=0$ independent of $\tau$ using the same parameter; so, $\Lambda_{z_j}(z(\tau))=c_j$ and $\frac{dz_j(\tau)}{d\tau}=z_j(\tau)+2c_j$ for $j=1,2,\ldots,n$ by equation \eqref{Eq4.1}.
Hence,
\begin{equation*}
z_j(\tau)=(d^*_j+2c_j)e^\tau-2c_j
\end{equation*}
with $d^*_j:=z_j(0)$ for $j=1,2,\ldots,n$.
Finally, equation \eqref{Eq4.3} implies
\begin{equation*}
\begin{split}
\frac{d\Lambda(\tau)}{d\tau}&=c_1z_1(\tau)+c_2z_2(\tau)+\cdots+c_nz_n(\tau)+2c^2_1+2c^2_2+\cdots+2c^2_n\\
&=c_1(d^*_1+2c_1)e^\tau+c_2(d^*_2+2c_2)e^\tau+\cdots+c_n(d^*_n+2c_n)e^\tau,
\end{split}
\end{equation*}
which leads to
\begin{equation*}
\Lambda(\tau)=c_1(d^*_1+2c_1)(e^\tau-1)+c_2(d^*_2+2c_2)(e^\tau-1)+\cdots+c_n(d^*_n+2c_n)(e^\tau-1)+\Lambda_0
\end{equation*}
with $\Lambda_0:=\Lambda(z(0))$, so that
\begin{equation}\label{Eq4.12}
\Lambda(z)=c_1z_1+c_2z_2+\cdots+c_nz_n+\Lambda^*(z)
\end{equation}
with $\Lambda^*(z)$ an entire function in $\C^n$ depending on $c_j,d^*_j$ for $j=1,2,\ldots,n$.
Suppose, without loss of generality, $\Lambda^*(z)$ has no linear terms that can be easily achieved from absorbing those terms into $c_1z_1+c_2z_2+\cdots+c_nz_n$, if necessary.
Then, one has
\begin{equation}\label{Eq4.13}
\Lambda^*(z)=c_0+\sum_{1\leq j\leq k\leq n}c_{jk}z_jz_k+\mathrm{terms}\,\,\mathrm{of}\,\,(z^3)\,\,\mathrm{or}\,\,\mathrm{higher}
\end{equation}
by abuse of notation of the term $z^3$ and
\begin{equation}\label{Eq4.14}
\sum^n_{j=1}\bigl[z_j\Lambda^*_{z_j}(z)+\bigl(c_j+\Lambda^*_{z_j}(z)\bigr)^2\bigr]-\Lambda^*(z)=0.
\end{equation}
Apply equation \eqref{Eq4.2} to \eqref{Eq4.14} to derive $\frac{dD\Lambda^*(\tau)}{d\tau}=0$ along any parametric curve/path, which together with \eqref{Eq4.13} yields $\Lambda^*(z)=c_0$.
So, equations \eqref{Eq4.10} and \eqref{Eq4.12} lead to
\begin{equation*}
u(z)=\Bigl(\frac{z_1}{2}+c_1\Bigr)^2+\Bigl(\frac{z_2}{2}+c_2\Bigr)^2+\cdots+\Bigl(\frac{z_n}{2}+c_n\Bigr)^2
\end{equation*}
in view of $c^2_1+c^2_2+\cdots+c^2_n=c_0$ by virtue of \eqref{Eq4.13} and \eqref{Eq4.14}.

\vskip 2pt
{\bf Case 4.} $\hbar=\ell$.
Now, the last equation in \eqref{Eq4.4} reads $\frac{du(\tau)}{d\tau}=\ell u^\ell$ so that
\begin{equation*}
\frac{1}{u^{\ell-1}(\tau)}=-\ell(\ell-1)\tau+\frac{1}{u_0^{\ell-1}}
\end{equation*}
with $u_0:=u(z(0))=\varphi(d_1,d_2,\ldots,d_n)$, which then yields
\begin{equation}\label{Eq4.15}
u(\tau)=\frac{u_0}{\bigl(1-\ell(\ell-1)u^{\ell-1}_0\tau\bigr)^{\frac{1}{\ell-1}}}.
\end{equation}
Equation \eqref{Eq4.15} combined with the second equation in \eqref{Eq4.4} further leads to
\begin{equation*}
\frac{du_{z_j}(z(\tau))}{d\tau}=\ell u^{\ell-1}u_{z_j}(z(\tau))=\frac{\ell u^{\ell-1}_0}{1-\ell(\ell-1)u^{\ell-1}_0\tau}u_{z_j}(z(\tau)),
\end{equation*}
so that
\begin{equation*}
u_{z_j}(z(\tau))=\frac{\varsigma_j}{\bigl(1-\ell(\ell-1)u^{\ell-1}_0\tau\bigr)^{\frac{1}{\ell-1}}}
\end{equation*}
with $\varsigma_j:=u_{z_j}(z(0))$ for $j=1,2,\ldots,n$.
Finally, one observes
\begin{equation*}
\frac{dz_j(\tau)}{d\tau}=\ell u^{\ell-1}_{z_j}(z(\tau))=\frac{\ell\varsigma^{\ell-1}_j}{1-\ell(\ell-1)u^{\ell-1}_0\tau}
\end{equation*}
using the first equation in \eqref{Eq4.4}, which implies
\begin{equation*}
z_j(\tau)=\frac{\varsigma^{\ell-1}_j}{u^{\ell-1}_0}\ln\Biggl(\frac{1}{\bigl(1-\ell(\ell-1)u^{\ell-1}_0\tau\bigr)^{\frac{1}{\ell-1}}}\Biggr)+d_j,
\end{equation*}
or, in a more convenient form for the purpose of comparing with \eqref{Eq4.15},
\begin{equation}\label{Eq4.16}
\frac{1}{\bigl(1-\ell(\ell-1)u^{\ell-1}_0\tau\bigr)^{\frac{1}{\ell-1}}}=e^{\frac{u^{\ell-1}_0}{\varsigma^{\ell-1}_j}(z_j(\tau)-d_j)}
\end{equation}
with $d_j:=z_j(0)$ for $j=1,2,\ldots,n$.
Therefore, by \eqref{Eq4.15} and \eqref{Eq4.16}, we have
\begin{equation*}
u(\tau)=\frac{u_0}{\bigl(1-\ell(\ell-1)u^{\ell-1}_0\tau\bigr)^{\frac{\sum^n_{j=1}\varrho_j}{\ell-1}}}\\
=u_0\prod^n_{j=1}\exp\biggl(\frac{\varrho_ju_0^{\ell-1}}{\varsigma^{\ell-1}_j}(z_j(\tau)-d_j)\biggr)
\end{equation*}
with $\varrho_1+\varrho_2+\cdots+\varrho_n=1$, which combined with $\varsigma^\ell_1+\varsigma^\ell_2+\cdots+\varsigma^\ell_n=u^\ell_0$ leads to
\begin{equation}\label{Eq4.17}
\begin{split}
u(z)&=u_0\exp\biggl(\sum^n_{j=1}\varrho_jz_j\biggl(\frac{\varsigma^\ell_1+\varsigma^\ell_2+\cdots+\varsigma^\ell_n}{\varsigma^\ell_j}\biggr)^{\frac{\ell-1}{\ell}}
+\Lambda_0(z)\biggr)\\
&=u_0\exp\biggl(\frac{\varsigma_1z_1+\varsigma_2z_2+\cdots+\varsigma_nz_n}{\sqrt[\ell]{\varsigma^\ell_1+\varsigma^\ell_2+\cdots+\varsigma^\ell_n}}
+\Lambda_1(z)\biggr)\\
&=u_0\exp\bigl(\sigma_1z_1+\sigma_2z_2+\cdots+\sigma_nz_n+\Lambda_2(z)\bigr)\\
&=\Psi(z)\exp(\sigma_1z_1+\sigma_2z_2+\cdots+\sigma_nz_n)
\end{split}
\end{equation}
by taking $\varrho_j:=\frac{\varsigma^\ell_j}{\varsigma^\ell_1+\varsigma^\ell_2+\cdots+\varsigma^\ell_n}$ and $\sigma_j:=\frac{\varsigma_j}{\sqrt[\ell]{\varsigma^\ell_1+\varsigma^\ell_2+\cdots+\varsigma^\ell_n}}$ for $j=1,2,\ldots,n$, with $\Lambda_{\mu}(z),\Psi(z)$ entire functions in $\C^n$ depending on $d_j,\varsigma_j$ for $j=1,2,\ldots,n$ and $\mu=0,1,2$.

It is worthwhile to note $u_0=\sqrt[\ell]{\varsigma^\ell_1+\varsigma^\ell_2+\cdots+\varsigma^\ell_n}$ can be a nontrivial entire function having zeros in $\C^n$; when laid in a quotient form, we implicitly meant all zeros in the numerator and denominator cancelled out, except for an analytic subset of $\C^n$ of codimension at least $2$.
So, $\varrho_j,\sigma_j\in\C$ were defined via the constant terms in the Taylor expansions of $u_0$ and $\varsigma_j$ over $\C^n$ for $j=1,2,\ldots,n$ respectively (by the same notations as those used to denote them as entire functions), with consensus that the remaining terms were merged into $\Lambda_0(z),\Lambda_1(z),\Lambda_2(z)$ and finally into $\Psi(z):=u_0(z)\exp(\Lambda_2(z))$ such that $\sum^n_{j=1}\sum^\ell_{\iota=1}(\sigma_j\Psi)^{\ell-\iota}\Psi^\iota_{z_j}=0$.

On the other hand, when $0$ is the finite Picard value of $\Psi(z)$, we can write $\Phi(z):=\ln(\Psi(z))$ to have $\sum^n_{j=1}\sum^\ell_{\iota=1}\sigma^{\ell-\iota}_j\Phi^\iota_{z_j}=0$ from \eqref{Eq1.4} through routine calculations and
\begin{equation*}
u(z)=\exp\bigl(\sigma_1z_1+\sigma_2z_2+\cdots+\sigma_nz_n+\Phi(z)\bigr).
\end{equation*}

Finally, we discuss {\bf Case 3}, whose proof follow those of {\bf Cases 2\&4} closely.
In particular, we refer to the preceding discussions when defining $\varrho_j,\sigma_j\in\C$ and require additionally that $u_0$ be an entire function in $\C^n$ such that $u^{\frac{\hbar}{\ell}}_0$ is also entire in $\C^n$.

\vskip 2pt
{\bf Case 3.} $1\leq\hbar<\ell$.
In this case, by the last equation in \eqref{Eq4.4}, one has
\begin{equation}\label{Eq4.18}
\begin{split}
u(\tau)&=u_0e^{\ell\tau}\hspace{30.09mm}\mathrm{if}\,\,\hbar=1\\
u(\tau)&=\frac{u_0}{\bigl(1-\ell(\hbar-1)u^{\hbar-1}_0\tau\bigr)^{\frac{1}{\hbar-1}}}\,\,\mathrm{if}\,\,1<\hbar<\ell
\end{split}
\end{equation}
with $u_0:=u(z(0))=\varphi(d_1,d_2,\ldots,d_n)$, which, by the second equation in \eqref{Eq4.4}, further imply
\begin{equation*}
\begin{split}
u_{z_j}(z(\tau))&=\varsigma_je^\tau\hspace{35.06mm}\mathrm{if}\,\,\hbar=1\\
u_{z_j}(z(\tau))&=\frac{\varsigma_j}{\bigl(1-\ell(\hbar-1)u^{\hbar-1}_0\tau\bigr)^{\frac{\hbar}{\ell(\hbar-1)}}}\,\,\mathrm{if}\,\,1<\hbar<\ell
\end{split}
\end{equation*}
with $\varsigma_j:=u_{z_j}(z(0))$ for $j=1,2,\ldots,n$.
Recalling
\begin{equation*}
\begin{split}
\frac{dz_j(\tau)}{d\tau}&=\ell\varsigma^{\ell-1}_je^{(\ell-1)\tau}\hspace{22.86mm}\mathrm{if}\,\,\hbar=1\\
\frac{dz_j(\tau)}{d\tau}&=\frac{\ell\varsigma^{\ell-1}_j}{\bigl(1-\ell(\hbar-1)u^{\hbar-1}_0\tau\bigr)^{\frac{\hbar(\ell-1)}{\ell(\hbar-1)}}}\,\,\mathrm{if}\,\,1<\hbar<\ell
\end{split}
\end{equation*}
by the first equation in \eqref{Eq4.4}, we deduce that
\begin{equation*}
\begin{split}
z_j(\tau)&=\frac{\ell}{\ell-1}\varsigma^{\ell-1}_j\bigl(e^{(\ell-1)\tau}-1\bigr)+d_j\hspace{36.89mm}\mathrm{if}\,\,\hbar=1 \medskip\\
z_j(\tau)&=\frac{\ell\varsigma^{\ell-1}_j}{(\ell-\hbar)u^{\hbar-1}_0}
\Biggl(\frac{1}{\bigl(1-\ell(\hbar-1)u^{\hbar-1}_0\tau\bigr)^{\frac{\ell-\hbar}{\ell(\hbar-1)}}}-1\Biggr)+d_j\,\,\mathrm{if}\,\,1<\hbar<\ell
\end{split}
\end{equation*}
and, in a more convenient form for the comparison with \eqref{Eq4.18}, that
\begin{equation}\label{Eq4.19}
\begin{split}
&e^{\tau}=\biggl(\frac{(z_j(\tau)-d_j)(\ell-1)}{\ell\varsigma^{\ell-1}_j}+1\biggr)^\frac{1}{\ell-1}\hspace{42.52mm}\mathrm{if}\,\,\hbar=1 \medskip\\
&\frac{1}{\bigl(1-\ell(\hbar-1)u^{\hbar-1}_0\tau\bigr)^{\frac{1}{\hbar-1}}}
=\biggl(\frac{(z_j(\tau)-d_j)(\ell-\hbar)u^{\hbar-1}_0}{\ell\varsigma^{\ell-1}_j}+1\biggr)^\frac{\ell}{\ell-\hbar}\,\,\mathrm{if}\,\,1<\hbar<\ell
\end{split}
\end{equation}
with $d_j:=z_j(0)$ for $j=1,2,\ldots,n$.

Now, when $\hbar=1$, the first equations in \eqref{Eq4.18} and \eqref{Eq4.19} yield $\ell=2$, and the first equation in \eqref{Eq4.18} can be further rewritten as
\begin{equation*}
u(\tau)=u_0\biggl(\sum^n_{j=1}\varrho_je^\tau\biggr)^2=u_0\biggl(\sum^n_{j=1}\varrho_j\biggl(\frac{z_j(\tau)-d_j}{2\varsigma_j}+1\biggr)\biggr)^2,
\end{equation*}
which further implies that, seeing $u_0=\varsigma^2_1+\varsigma^2_2+\cdots+\varsigma^2_n$,
\begin{equation}\label{Eq4.20}
\begin{split}
u(z)&=u_0\biggl(\frac{1}{2}\biggl(\frac{\varrho_1z_1}{\varsigma_1}+\frac{\varrho_2z_2}{\varsigma_2}+\cdots+\frac{\varrho_nz_n}{\varsigma_n}\biggr)+\Lambda_0(z)\biggr)^2\\
&=\biggl(\frac{1}{2}\frac{\varsigma_1z_1+\varsigma_2z_2+\cdots+\varsigma_nz_n}{\varsigma^2_1+\varsigma^2_2+\cdots+\varsigma^2_n}
\bigl(\varsigma^2_1+\varsigma^2_2+\cdots+\varsigma^2_n\bigr)^\frac{1}{2}+\Lambda_1(z)\biggr)^2\\
&=\Bigl(\frac{1}{2}(\sigma_1z_1+\sigma_2z_2+\cdots+\sigma_nz_n)+\Phi(z)\Bigr)^2
\end{split}
\end{equation}
by taking $\varrho_j:=\frac{\varsigma^2_j}{\varsigma^2_1+\varsigma^2_2+\cdots+\varsigma^2_n}$ and $\sigma_j:=\frac{\varsigma_j}{\sqrt{\varsigma^2_1+\varsigma^2_2+\cdots+\varsigma^2_n}}$ for $j=1,2,\ldots,n$, with $\Lambda_\mu(z),\Phi(z)$ entire functions in $\C^n$ depending on $d_j,\varsigma_j$ for $j=1,2,\ldots,n$ and $\mu=0,1$.

Next, when $1<\hbar<\ell$, the second equations in \eqref{Eq4.18} and \eqref{Eq4.19} show that $\frac{\ell}{\ell-\hbar}$ needs to be an integer, and the second equation in \eqref{Eq4.18} can be further rewritten as
\begin{equation*}
\begin{split}
u(\tau)&=u_0\Biggl(\sum^n_{j=1}\varrho_j\frac{1}{\bigl(1-\ell(\hbar-1)u^{\hbar-1}_0\tau\bigr)^{\frac{\ell-\hbar}{\ell(\hbar-1)}}}\Biggr)^{\frac{\ell}{\ell-\hbar}}\\
&=u_0\biggl(\sum^n_{j=1}\varrho_j\biggl(\frac{(z_j(\tau)-d_j)(\ell-\hbar)u^{\hbar-1}_0}{\ell\varsigma^{\ell-1}_j}+1\biggr)\biggr)^{\frac{\ell}{\ell-\hbar}},
\end{split}
\end{equation*}
which further implies that, seeing $u^\hbar_0=\varsigma^\ell_1+\varsigma^\ell_2+\cdots+\varsigma^\ell_n$,
\begin{equation}\label{Eq4.21}
\begin{split}
u(z)&=u_0\biggl(\frac{\ell-\hbar}{\ell}\biggl(\frac{\varrho_1z_1}{\varsigma^{\ell-1}_1}+\frac{\varrho_2z_2}{\varsigma^{\ell-1}_2}
+\cdots+\frac{\varrho_nz_n}{\varsigma^{\ell-1}_n}\biggr)u^{\hbar-1}_0+\Lambda_0(z)\biggr)^{\frac{\ell}{\ell-\hbar}}\\
&=\biggl(\frac{\ell-\hbar}{\ell}
\frac{\varsigma_1z_1+\varsigma_2z_2+\cdots+\varsigma_nz_n}{\varsigma^\ell_1+\varsigma^\ell_2+\cdots+\varsigma^\ell_n}u^{\hbar-1+\frac{\ell-\hbar}{\ell}}_0
+\Lambda_1(z)\biggr)^{\frac{\ell}{\ell-\hbar}}\\
&=\biggl(\frac{\ell-\hbar}{\ell}\frac{\varsigma_1z_1+\varsigma_2z_2+\cdots+\varsigma_nz_n}{\varsigma^\ell_1+\varsigma^\ell_2+\cdots+\varsigma^\ell_n}
\bigl(\varsigma^\ell_1+\varsigma^\ell_2+\cdots+\varsigma^\ell_n\bigr)^{1-\frac{1}{\ell}}+\Lambda_1(z)\biggr)^{\frac{\ell}{\ell-\hbar}}\\
&=\Bigl(\frac{\ell-\hbar}{\ell}(\sigma_1z_1+\sigma_2z_2+\cdots+\sigma_nz_n)+\Phi(z)\Bigr)^{\frac{\ell}{\ell-\hbar}}
\end{split}
\end{equation}
by taking $\varrho_j:=\frac{\varsigma^\ell_j}{\varsigma^\ell_1+\varsigma^\ell_2+\cdots+\varsigma^\ell_n}$ and $\sigma_j:=\frac{\varsigma_j}{\sqrt[\ell]{\varsigma^\ell_1+\varsigma^\ell_2+\cdots+\varsigma^\ell_n}}$ for $j=1,2,\ldots,n$, with $\Lambda_\mu(z),\Phi(z)$ entire functions in $\C^n$ depending on $d_j,\varsigma_j$ for $j=1,2,\ldots,n$ and $\mu=0,1$.

All the preceding discussions conclude the proof of Theorem \ref{Thm8}.
\end{proof}

\vskip 6pt
{\small{\bf Acknowledgement.} The author wholeheartedly thanks the editor and the anonymous reviewers for valuable suggestions.
The author acknowledges Dr. Wei Chen of providing the references \cite{LT,LZ} and reading the initial proof of Theorem \ref{Thm1} with helpful comments, and Dr. Jingbo Liu of providing the reference \cite{BMS} and discussing the examples and the proof of Theorem \ref{Thm8}, {\bf Cases 2-4}.}


\vskip 6pt
\begin{thebibliography}{28}
\addcontentsline{toc}{chapter}{Bibliography}

\bibitem{BK1} S.B. Bank and R.P. Kaufman. On meromorphic solutions of first-order differential equations. {\sl Comment. Math. Helv.} {\bf51} (1976), 289-299.

\bibitem{BK2} S.B. Bank and R.P. Kaufman. On the order of growth of meromorphic solutions of first-order differential equations. {\sl Math. Ann.} {\bf241} (1979), 57-67.

\bibitem{BMS} M. Bennett, P. Mih\u{a}ilescu, and S. Siksek. {\sf The generalized Fermat equation}. Open Problems in Mathematics (Edited by John F. Nash, Jr. and
Michael Th. Rassias), 173-205. Springer, Cham (2016).

\bibitem{Be1} W. Bergweiler. Fixed points of composite meromorphic functions and normal families. {\sl Proc. Roy. Soc. Edinburgh Sect. A} {\bf134} (2004), 653-660.

\bibitem{Be2} W. Bergweiler. Fixed points of composite entire and quasiregular maps. {\sl Ann. Acad. Sci. Fenn. Math.} {\bf31} (2006), 523-540.

\bibitem{CC} L.A. Caffarelli and M.G. Crandall. Distance functions and almost global solutions of eikonal equations. {\sl Comm. Partial Differential Equations} {\bf35} (2010), 391-414.

\bibitem{CLY} D.C. Chang, B.Q. Li, and C.C. Yang. On composition of meromorphic functions in several complex variables. {\sl Forum Math.} {\bf7} (1995), 77-94.

\bibitem{Ev} L.C. Evans. {\bf Partial Differential Equations}. American Mathematical Society, Providence, RI (2010).

\bibitem{Gr} F. Gross. On factorization of meromorphic functions. {\sl Trans. Amer. Math. Soc.} {\bf120} (1965), 124-144.

\bibitem{GH} G.G. Gundersen and W.K. Hayman. The strength of Cartan's version of Nevanlinna theory. {\sl Bull. London Math. Soc.} {\bf36} (2004), 433-454.

\bibitem{Ha} Q. Han. On complex analytic solutions of the partial differential equation $(u_{z_1})^m+(u_{z_2})^m=u^m$. {\sl Houston J. Math.} {\bf35} (2009), 277-289.

\bibitem{Ha1} W.K. Hayman. {\bf Meromorphic Functions}. Oxford University Press, Oxford (1964).

\bibitem{Ha2} W.K. Hayman. {\sf Waring's problem f\"{u}r analytische funktionen}. Mathematisch-Naturwissenschaftliche Klasse Sitzungsberichte 1984, 1-13. Bayerische Akademie der Wissenschaften, M\"{u}nchen (1985).

\bibitem{Ha3} W.K. Hayman. Waring's theorem and the super Fermat problem for numbers and functions. {\sl Complex Var. Elliptic Equ.} {\bf59} (2014), 85-90.

\bibitem{He} J.E. Hemmati. Entire solutions of first-order nonlinear partial differential equations. {\sl Proc. Amer. Math. Soc.} {\bf125} (1997), 1483-1485.

\bibitem{IT} K. Ishizaki and N. Toda. Transcendental meromorphic solutions of some algebraic differential equations. {\sl J. Aust. Math. Soc.} {\bf83} (2007), 157-180.

\bibitem{Jo} G. Johnsson. The Cauchy problem in $\C^N$ for linear second order partial differential equations with data on a quadric surface. {\sl Trans. Amer. Math. Soc.} {\bf344} (1994), 1-48.

\bibitem{Kh} D. Khavinson. A note on entire solutions of the eikonal equation. {\sl Amer. Math. Monthly} {\bf102} (1995), 159-161.

\bibitem{Li1} B.Q. Li. On entire solutions of Fermat type partial differential equations. {\sl Internat. J. Math.} {\bf15} (2004), 473-485.

\bibitem{Li2} B.Q. Li. Entire solutions of certain partial differential equations in $\C^n$. {\sl Israel J. Math.} {\bf143} (2004), 131-140.

\bibitem{Li3} B.Q. Li. Entire solutions of certain partial differential equations and factorization of partial derivatives. {\sl Trans. Amer. Math. Soc.} {\bf357} (2005), 3169-3177.

\bibitem{Li4} B.Q. Li. On meromorphic solutions of $f^2+g^2=1$. {\sl Math. Z.} {\bf258} (2008), 763-771.

\bibitem{Li5} B.Q. Li. On certain non-linear differential equations in complex domains. {\sl Arch. Math. (Basel)} {\bf91} (2008), 344-353.

\bibitem{Li6}B.Q. Li. On meromorphic solutions of generalized Fermat equations. {\sl Internat. J. Math.} {\bf25} (2014), 1450002, 8 pp.

\bibitem{LS} B.Q. Li and E.G. Saleeby. Entire solutions of first-order partial differential equations. {\sl Complex Var. Theory Appl.} {\bf48} (2003), 657-661.

\bibitem{LY} B.Q. Li and C.C. Yang. {\sf Factorization of meromorphic functions in several complex variables}. Contemp. Math. {\bf142}, 61-74. American Mathematical Society, Providence, RI (1993).

\bibitem{LYe} B.Q. Li and Z. Ye. On meromorphic solutions of $f^3+g^3=1$. {\sl Arch. Math. (Basel)} {\bf90} (2008), 39-43.

\bibitem{LT} L. Liao and J. Tang. The transcendental meromorphic solutions of a certain type of nonlinear differential equations. {\sl J. Math. Anal. Appl.} {\bf334} (2007), 517-527.

\bibitem{LZ} L. Liao and X. Zhang. On a certain type of nonlinear differential equations admitting transcendental meromorphic solutions. {\sl Sci. China Math.} {\bf56} (2013), 2025-2034.

\bibitem{Sa1} E.G. Saleeby. Entire and meromorphic solutions of Fermat type partial differential equations. {\sl Analysis (Munich)} {\bf19} (1999), 369-376.

\bibitem{Sa2} E.G. Saleeby. On entire and meromorphic solutions of $\lambda u^k+\sum^n_{i=1}u^m_{z_i}=1$. {\sl Complex Var. Theory Appl.} {\bf49} (2004), 101-107.

\bibitem{Vi} A. Vitter. The lemma of the logarithmic derivative in several complex variables. {\sl Duke Math. J.} {\bf44} (1977), 89-104.

\end{thebibliography}
\end{document}